\documentclass[12 pt]{article}%
\usepackage{amsmath, amsfonts, amsthm, color,latexsym}
\usepackage{amsmath, ulem}
\usepackage{amsfonts}
\usepackage{amssymb}
\usepackage{color, soul}
\usepackage{xcolor}
\usepackage{tikz-cd}
\usepackage[all]{xy}
\usepackage{graphicx}%
\providecommand{\U}[1]{\protect\rule{.1in}{.1in}}
\allowdisplaybreaks[4]
\newtheorem{theorem}{Theorem}[section]
\newtheorem{proposition}[theorem]{Proposition}
\newtheorem{corollary}[theorem]{Corollary}
\newtheorem{example}[theorem]{Example}
\newtheorem{examples}[theorem]{Examples}

\newtheorem{remarks}[theorem]{Remarks}
\newtheorem{lemma}[theorem]{Lemma}
\newtheorem{final remark}[theorem]{Final Remark}
\newtheorem{definition}[theorem]{Definition}
\allowdisplaybreaks[4]

\newcommand {\R}{\mathbb{R}}
\newcommand {\K} {\mathbb{K}}

\newcommand {\N} {\mathbb{N}}
\newcommand{\norma}[1]{\| #1 \|}
\newcommand{\conj}[2]{\left \{ {#1} \, : \, {#2} \right \}}

\newcommand{\inner}[2]{\langle {#1},{#2} \rangle}
\begin{document}

\title{A $p$-summability approach to the weak maximizing property}
\author{Oscar I. Blanco A.\thanks{Supported by a CAPES PhD scholarship (Grant 88887.001877/2024-00)}  \, and Vinícius C. C. Miranda\thanks{Supported by FAPESP (Grants 2025/08630-0 and 2023/12916-1) and FAPEMIG (Grant APQ-02622-26) \newline 2020 Mathematics Subject Classification:  46B20, 46B15, 46B25.
}}
\date{}
\maketitle

\begin{abstract} 
Motivated by the weak maximizing property ($\mathrm{WMP}$), we investigate $p$-summability conditions on maximizing sequences of bounded linear operators. Since the naive $p$-summability formulation is independent of $p$ and collapses to norm attainment for all operators, we introduce the weakly $p$-singular maximizing property ($\mathrm{SMP}_p$), based on maximizing sequences with no weakly $p$-summable subsequence. We completely characterize when the pairs $(\ell_p,\ell_q)$ and $(\ell_p,c_0)$ have the $\mathrm{SMP}_r$, revealing sharp contrasts with the $\mathrm{WMP}$. We also characterize the Schur property of order $p$ via the $\mathrm{WMP}$ and the $\mathrm{SMP}_r$. Under relative weak $p$-precompactness and suitable Dunford-Pettis-type assumptions, we characterize universal norm attainment for operators from $X$ to $Y$ in terms of either the $\mathrm{SMP}_q$ for adjoint operators or the
weak$^{*}$-to-weak$^{*}$ maximizing property, and derive corresponding duality consequences for the $\mathrm{SMP}_p$. Finally, we introduce
$p$-convergent perturbation properties for operators and their adjoints, characterize them for classical sequence spaces, and show that the $p$-convergent perturbation property is strictly weaker than
the $\mathrm{SMP}_p$.
\end{abstract}

\section{Introduction}

In \cite{danieleduardo}, Pellegrino and Teixeira proved that, for each $1 < p < \infty$ and each $1 \leq q < \infty$, every bounded linear $T: \ell_p \to \ell_q$ attains its norm if and only if there exists a non-weakly null maximizing sequence $(x_n)_n \subset S_{\ell_p}$. Recall that a maximizing sequence for a bounded linear operator $T$ is a sequence $(x_n)_n \subset S_X$ such that $\displaystyle \lim_{n \to \infty} \norma{T(x_n)} = \norma{T}$.  This result motivated the introduction of the so-called weak maximizing property, defined by Aron, García, Pellegrino and Teixeira in \cite{arongarciapelteix}: a pair of Banach spaces $(X, Y)$ satisfies the {\bf weak maximizing property} ($\mathrm{WMP}$, in short) if a bounded linear operator $T: X \to Y$ is norm-attaining whenever there exists a non-weakly null maximizing sequence for $T$. Recently, this property has been also studied in \cite{dantasjung, garcia-lirola, jung, miranda}. We also note that variants of the property were studied by Chakraborty \cite{cha} concerning the minimum norm, and by Luiz and the second named author in the setting of Banach lattices \cite{luizmiranda}.

Following a well-established line of investigation in the theory of Banach spaces that seeks to interpret properties by replacing weak convergence by weak $p$-summability (see, e.g., \cite{castillo, P-Schur, zeefou14, coarse, karnsinha2, sinhakarn}), it is natural to ask whether the weak maximizing property admits a meaningful $p$-summability version. More precisely, one could replace non-weakly null maximizing sequences by maximizing sequences which are not weakly $p$-summable, for some $1\leq p<\infty$. As we shall see, however, this first attempt does not lead to a genuinely new property.

We first show that the particular choice of $p$ does not affect the resulting property:

\begin{proposition}\label{primeiro}
Let $X$ and $Y$ be Banach spaces, let $T\in\mathcal L(X;Y)$, and let
$1\leq p<\infty$. Suppose that $T$ admits a maximizing sequence
$(x_n)_n\subset S_X$ which is not weakly $p$-summable. Then, for every
$1\leq q<\infty$, $T$ admits a maximizing sequence which is
not weakly $q$-summable.
\end{proposition}

\begin{proof}
Since $(x_n)_n\notin \ell_p^w(X)$, there exists $x^*\in X^*$ such that $(x^*(x_n))_n\notin \ell_p.$
In particular, the set
$A:=\conj{n \in \N}{x^*(x_n) \neq 0}$
is infinite. Fix $1\leq q<\infty$. We construct the sequence $(m_n)_n \subset \N$ as follows: for each $n\in A$, choose $m_n\in\N$ such that  $m_n |x^*(x_n)|^q\geq 1$. If $n\notin A$, choose $m_n=1$. Define
$$ (z_k)_k   =  (\underbrace{x_1,\ldots,x_1}_{m_1\text{ times}},   \underbrace{x_2,\ldots,x_2}_{m_2\text{ times}},  \ldots,  \underbrace{x_n,\ldots,x_n}_{m_n\text{ times}},    \ldots) \subset S_X.$$
We claim that $(z_k)_k$ is a maximizing sequence for $T$. Indeed, since
$(x_n)_n$ is a maximizing sequence for $T$, given $\varepsilon > 0$, there exists $n_0 \in \N$ such that $\norma{Tx_n} > \norma{T} - \varepsilon$ for every $n \geq n_0$. Taking $k_0 = m_1 + \cdots + m_{n_0 -1} + 1$, by the definition of $(z_k)_k$, we have that for each $k \geq k_0$, $z_k = x_n$ for some $n \geq n_0$. Hence $\norma{Tz_k} > \norma{T} - \varepsilon$ for every $k \geq k_0$, proving that $\displaystyle \lim_{k \to \infty} \norma{Tz_k} = \norma{T}$.
On the other hand,
$$    \sum_{k=1}^{\infty} |x^*(z_k)|^q
    =
    \sum_{n=1}^{\infty} m_n |x^*(x_n)|^q
    \geq
    \sum_{n\in A} 1
    =
    \infty.$$
Thus $(x^*(z_k))_k\notin \ell_q$, and consequently,  $(z_k)_k\notin \ell_q^w(X).$
\end{proof}

Proposition \ref{primeiro} shows that all finite weak $p$-summable versions of the $\mathrm{WMP}$ obtained in this way would coincide. The next result shows that the situation is even more degenerate: this naive formulation is equivalent to requiring that every bounded linear operator from $X$ into $Y$ attain its norm.

\begin{theorem}
Let $X$ and $Y$ be Banach spaces with $Y\neq{0}$. The following assertions are equivalent: \\
{\rm (1)} There exists $1\leq p<\infty$ such that every bounded linear operator $T:X\to Y$ with a maximizing sequence which is not weakly $p$-summable is norm-attaining. \\
{\rm (2)} For every $1\leq p<\infty$, every bounded linear operator $T:X\to Y$ admitting a maximizing sequence which is not weakly $p$-summable is norm-attaining. \\
{\rm (3)} Every bounded linear operator $T:X\to Y$ is norm-attaining.
\end{theorem}

\begin{proof}
The implications (3)$\Rightarrow$(2) and (2)$\Rightarrow$(1) are immediate.

Assume now that (1) holds for some $1 \leq p < \infty$. We first observe that it implies that $(X, Y)$ has the $\mathrm{WMP}$. Indeed, if $T \in \mathcal{L}(X; Y)$ admits a non-weakly null maximizing sequence, then this sequence cannot be weakly $p$-summable. Hence, by our assumption, $T$ attains its norm. In particular, as $Y \neq 0$, it follows from \cite[Corollary 2.5]{arongarciapelteix} that $X$ is reflexive. We now prove (3). For the sake of contradiction, we suppose that there exists a non-norm attaining operator $0 \neq T: X \to Y$. Let $(x_n)_n$ be a maximizing sequence for $T$. Since $(X, Y)$ has the $\mathrm{WMP}$, 
$(x_n)_n$ must be weakly null. Thus, it follows from \cite[Proposition 1.5.4]{albiac} that $(x_n)_n$ contains a basic subsequence $(x_{n_k})_k$. We claim that $(x_{n_k})_k$ is not weakly $1$-summable. Indeed, assuming otherwise, the bounded linear operator $$ J: (\overline{[x_{n_k} : k \in \N]})^* \to \ell_1, \qquad J(x^*) = (x^*(x_{n_k}))_k, $$
is well-defined. Moreover, $J$ is bounded by the Closed Graph Theorem, and $J$ is injective since $(x_{n_k})_k$ is a basic sequence. Therefore, $(\overline{[x_{n_k} : k \in \N]})^*$ is isomorphic to a subspace of $\ell_1$. 
However, since $Z = \overline{[x_{n_k} : k \in \N]}$ is a closed subspace of the reflexive space $X$, we have that $Z$ and $Z^*$ are reflexive spaces. This shows that $Z^*$ is an infinite-dimensional reflexive space that is isomorphic to a subspace of $\ell_1$, which is impossible since every subspace of $\ell_1$ has the Schur property. This proves that $(x_{n_k})_k$ is not weakly $1$-summable. 
Finally, recalling that $(x_{n_k})_k$ is also a maximizing sequence for $T$, we get from Proposition \ref{primeiro} that $T$ admits a maximizing sequence which is not weakly $p$-summable, so by the assumption $T$ attains its norm, a contradiction. 
\end{proof}

Thus, simply replacing non-weakly null maximizing sequences by maximizing sequences which are not weakly $p$-summable does not provide a meaningful $p$-summability refinement of the weak maximizing property. Indeed, the preceding result shows that this formulation does not distinguish a genuinely $p$-summability phenomenon from the much stronger requirement that all operators attain their norm. To obtain a nontrivial $p$-version, we require the obstruction to weak $p$-summability to persist along subsequences. We say that a sequence $(x_n)_n$ is weakly $p$-singular if it contains no weakly $p$-summable subsequence. This leads to the following definition.

%%%\corr{errado}{correção}

\begin{definition} Let $1 \leq p < \infty$. A pair of Banach spaces $(X, Y)$ is said to have the \textbf{weakly $p$-singular maximizing property} ($\mathrm{SMP}_p$, for short) if every bounded linear operator $T: X \to Y$ which admits a weakly $p$-singular maximizing sequence is norm-attaining.
\end{definition}

In Section 2 we show that this definition leads to a nontrivial hierarchy of properties. More precisely, $\mathrm{SMP}_p$ implies $\mathrm{SMP}_q$ whenever $p<q$, and the examples studied below show that the converse implications fail in a sharp way. We characterize the values of $r$ for which the pairs $(\ell_p,\ell_q)$ and $(\ell_p,c_0)$ have the $\mathrm{SMP}_r$. In particular, $(L_2[0,1],L_2[0,1])$ has the $\mathrm{SMP}_r$ if and only if $r\geq 2$: the failure for $r<2$ follows from our examples, while the positive part follows from a criterion based on the property strict $(M)$, adapted from ideas of Han and Kim \cite{han}. This criterion also yields further examples involving $\ell_p$-sums of finite dimensional spaces.

 In \cite{dantasjung}, Dantas, Jung and Martínez-Cervantes proved that a Banach space $Y$ has the Schur property if and only if the pair $(X, Y)$ has the $\mathrm{WMP}$ for every reflexive space $X$. When shifting our focus to the weak $p$-summable sequences, it is natural to seek a similar characterization. Recall that a Banach space $X$ is said to have the \textbf{Schur property of order \textit{p}} ($\mathrm{SP}_p$, for short) if every weakly $p$-summable sequence in $X$ is norm null. For recent developments regarding this property, we refer the reader to \cite{alikhani, P-Schur}. In our setting, the reflexivity of the domain space $X$ will be replaced by a natural topological condition: the weak $p$-precompactness of its closed unit ball $B_X$. A subset $K \subseteq X$ is \textbf{relatively weakly \textit{p}-precompact} if every sequence in $K$ admits a weakly $p$-convergent subsequence. The case $p=\infty$ recovers the classical relatively weakly compact sets. For instance, for every $1 < p < \infty$, $B_{\ell_p}$ is relatively weakly $p^*$-precompact, and the closed unit ball of $X_p = \ell_p(\oplus_{i=1}^\infty \ell_1^i)$ is relatively weakly $p^*$-precompact (see \cite{castillo, chenchavez}). It is important to observe that in \cite{castillo}, Castillo and Sánchez referred to this class of sets as relatively weakly $p$-compact sets. However, in order to avoid confusion with the notion of the same name introduced in \cite{sinhakarn}, we will follow the terminology adopted in \cite{chenchavez}. Armed with these notions, Theorem \ref{p-schuriff} characterizes Banach spaces with the $\mathrm{SP}_p$. More precisely, for a given $1<p<\infty$, a Banach space $Y$ has the $\mathrm{SP}_p$ if and only if every bounded linear operator $T:X\to Y$ is norm-attaining whenever $B_X$ is relatively weakly $p$-precompact. The theorem also provides equivalent formulations in terms of the $\mathrm{WMP}$ and the $\mathrm{SMP}_r$ for every $1\leq r<\infty$.

Motivated by the weak$^\ast$ maximizing properties introduced by García-Lirola and Petitjean in \cite{garcia-lirola}, we also investigate the $\mathrm{SMP}_p$ in the setting of adjoint operators. Since weakly $p$-summable and weak$^\ast$ $p$-summable sequences coincide in dual spaces (see, e.g., \cite[p. 2]{fourie}), replacing weakly $p$-singular sequences with weak$^\ast$ $p$-singular sequences would not give rise to a new version of the $\mathrm{SMP}_p$ on pairs of dual spaces. Nevertheless, restricting our attention to operators that are adjoints
gives rise to a formally weaker version of the $\mathrm{SMP}_p$. This leads to the following definition.

\begin{definition}  \label{defsmppadj}
Let $1\leq p<\infty$. A pair of dual Banach spaces $(X^\ast,Y^\ast)$ is said to have the $\mathrm{SMP}_p$ for adjoint operators if every adjoint operator $T^*: X^* \to Y^*$ admitting a weakly $p$-singular maximizing sequence is norm-attaining.
\end{definition}

In Section 3, we relate the $\mathrm{SMP}_p$ for adjoint operators to both the $\mathrm{SMP}_p$ on pairs of dual Banach spaces and the  weak$^*$-to-weak$^*$ maximizing property, and obtain a useful criterion based on the strict $(sM^\ast)$ property introduced in \cite{han}. 
We then establish results which, under natural weak precompactness and Dunford-Pettis-type assumptions, make the $\mathrm{SMP}_p$ of $(X,Y)$ equivalent to either the $\mathrm{SMP}_q$ for adjoint operators or the weak$^\ast$-to-weak$^\ast$ maximizing property of $(Y^\ast,X^\ast)$. These results provide new insight into \cite[Question 5.12]{garcia-lirola}, while our examples clarify the role of the hypotheses involved.

Finally, motivated by results from \cite{arongarciapelteix} and \cite{jung}, we now turn our attention to a perturbation property related to the weakly $p$-singular maximizing property. It follows from \cite{arongarciapelteix} that the classical weak maximizing property implies the compact perturbation property. Following \cite{jung}, a pair of Banach spaces $(X,Y)$ is said to have the \textbf{compact perturbation property} ($\mathrm{CPP}$, for short) if $T+K$ is norm-attaining whenever $T:X\to Y$ is a bounded linear operator and $K:X\to Y$ is a compact operator such that
$\|T\|<\|T+K\|.$

In the present $p$-summability setting, the natural class of perturbations is given by $p$-convergent operators. Recall that a bounded linear operator $K:X\to Y$ is said to be $p$-convergent if it maps weakly $p$-summable sequences in $X$ into norm null sequences in $Y$.

\begin{definition}\label{pcpp}
Let $1\leq p<\infty$. A pair of Banach spaces $(X,Y)$ is said to have the \textbf{$p$-convergent perturbation property} ($p$-$\mathrm{CPP}$, for short) if $T+K$ is norm-attaining whenever $T:X\to Y$ is a bounded linear operator and $K:X\to Y$ is a $p$-convergent operator such that $\|T\|<\|T+K\|.$
\end{definition}

For $p=\infty$, we adopt the convention that $\infty$-$\mathrm{CPP}$ coincides with the classical compact perturbation property. In Section 4, we prove that the $\mathrm{SMP}_p$ implies the $p$-$\mathrm{CPP}$, but the converse implication is not true. We provide a complete characterization when the pairs $(\ell_r, \ell_q)$ and $(\ell_r, c_0)$ enjoy this perturbation property, and we obtain a relatively $p$-precompact version of \cite[Theorem 2.1]{jung}. 

Motivated by the adjoint compact perturbation property introduced by Han and Kim in \cite{han}, we also consider an adjoint version of the $p$-CPP. For $1 \leq p < \infty$, the adjoint $p$-essential norm of $T \in \mathcal{L}(X; Y)$ is defined by
$$ \norma{T}_{e,p}^{\rm adj} := \inf \conj{\norma{T-K}}{K \in \mathcal{L}(X; Y) \text{ and } K^* \text{ is $p$-convergent}}. $$

\begin{definition}
\label{adpcpp}
Let $ 1 \leq p < \infty$. A pair of Banach spaces $(X, Y)$ is said to have the $p$-convergent perturbation property for adjoint operators ($p$-ACPP, for short) if $\norma{T}_{e,p}^{\rm adj} = \norma{T}$ for every $T \in \mathcal{L}(X; Y)$ whose adjoint $T^*$ does not attain its norm. 
\end{definition}

In Section 4, we obtain a pertubative characterization of the $p$-ACPP, relate it to the $p$-CPP and the $\mathrm{SMP}_p$ for adjoint operators, and examine the property for classical sequence spaces.

We refer the reader to \cite{albiac, geraldo, fabianhabala} for any unexplained terminology and standard facts concerning Banach space theory.

\section{The weak $p$-singular maximizing property}

We begin this section providing standard facts about the $\mathrm{SMP}_p$:

\begin{proposition}
\label{prop1}
Let $X$ and $Y$ be Banach spaces such that $(X, Y)$ has the $\mathrm{SMP}_p$ for some $1 \leq p < \infty$. \\
{\rm (1)} If $F$ is a closed subspace of $Y$, then $(X,F)$ has the $\mathrm{SMP}_p$. \\
{\rm (2)} If $E$ is a $1$-complemented subspace of $X$, then $(E,Y)$ has the $\mathrm{SMP}_p$. \\
{\rm (3)} If $q > p$, then $(X, Y)$ has the $\mathrm{SMP}_q$. \\
{\rm (4)} The pair $(X, Y)$ has the $\mathrm{WMP}$.
\end{proposition}

\begin{proof}
    The proofs of items (1) and (2) are standard and will be omitted (see, e.g., \cite[Proposition 2.2]{dantasjung} or \cite[Proposition 3.3]{luizmiranda}). 

    (3) Let $T \in \mathcal{L}(X;Y)$ and $(x_{n})_{n} \subset S_X$ be a weakly $q$-singular maximizing sequence for $T$. Since $q > p$, no subsequence of $(x_n)_n$ can be weakly $p$-summable, which implies that $(x_n)_n$ is weakly $p$-singular. Thus, $T$ attains its norm, proving that $(X, Y)$ has the $\mathrm{SMP}_q$.

    (4) Let  $T \in \mathcal{L}(X;Y)$ and $(x_{n})_{n} \subset S_X$ be a non-weakly null maximizing sequence for $T$. Since $(x_n)_n$ is not weakly null in $X$, there exists $x^* \in X^*$ such that $(x^*(x_n))_n$ does not converge to $0$ in $\K$. So, by passing to a subsequence if necessary, we may assume that $|x^*(x_n)| \geq \varepsilon$ holds for some $\varepsilon > 0$ and every $n \in \N$. Thus, $(x_n)_n$ does not have any weakly null subsequence, and consequently $(x_n)_n$ cannot have any weakly $p$-summable subsequence. Then, $(x_n)_n$ is weakly $p$-singular, so the hypothesis yields that $T$ attains its norm, which proves that $(X, Y)$ has the $\mathrm{WMP}$.
\end{proof}

It follows from Proposition \ref{prop1} that the following chain of implications hold for every pair $(X, Y)$ of Banach spaces and $1 \leq p < q < \infty$:
$$ (X, Y) \text{ has the } \mathrm{SMP}_p \Rightarrow (X, Y) \text{ has the } \mathrm{SMP}_q \Rightarrow (X, Y) \text{ has the } \mathrm{WMP}. $$
Our next result not only shows that the converse implications above are false as it characterizes when the pair $(\ell_p, \ell_q)$ has the $\mathrm{SMP}_r$.

\begin{theorem} \label{teoexlp}
    Let $1 \leq q, r < \infty$ and $1 < p < \infty$ be given. Then, the pair $(\ell_p, \ell_q)$ has the $\mathrm{SMP}_r$ if and only if $q< p$ or $r \geq p^*$.
\end{theorem}

\begin{proof}    If $q < p$, then every bounded linear operator from $\ell_p$ into $\ell_q$ is norm-attaining - this follows from Pitt's theorem and \cite[Theorem B]{sheldon}. Thus, $(\ell_p, \ell_q)$ has the $\mathrm{SMP}_r$ for every $r \geq 1$.  
Now, we assume that $r \geq p^*$. For the sake of contradiction, we assume that $(\ell_p, \ell_q)$ does not have the $\mathrm{SMP}_r$. So, there exists a non-norm attaining operator $T: \ell_p \to \ell_q$ and there exists weakly $r$-singular maximizing sequence $(x_n)_n \subset \ell_p$ for $T$. By \cite[Proposition 5]{danieleduardo}, $(x_n)_n$ has a subsequence $(x_{n_k})_k$
that is equivalent to the canonical basis $(e_k)_k$ of $\ell_p$.   Since $(e_k)_k$ is weakly $r$-summable in $\ell_p$ for every $r \geq p^*$, $(x_{n_k})_k$ is also   weakly $r$-summable for every $r \geq p^*$ contradicting the weakly $r$-singularity of $(x_n)_n$. Therefore, $(\ell_p, \ell_q)$ has the $\mathrm{SMP}_r$.

For the converse, we assume that $q \geq p$ and that $r < p^*$, and we prove that $(\ell_p, \ell_q)$ does not have the $\mathrm{SMP}_r$. Indeed, since $q \geq p$, we can consider non-norm attaining operator $T: \ell_{p} \to \ell_{q}$ defined by $\displaystyle T((a_{n})_{n})=\left (\frac{n a_{n}}{n+1}\right )_{n}$ (see \cite[p. 419]{danieleduardo} for the non-norm attainment of $T$). Notice that $(e_n)_n$ is a maximizing sequence for $T$ that is weakly $r$-singular since $r < p^*$.    
\end{proof}

The following is an application of Theorem  \ref{teoexlp} 

\begin{corollary} \label{corlp1} Let $2 \leq  p < \infty$ and $1 < r < \infty$ be given. The pair $(\ell_p, L_p([0,1]))$  has the $\mathrm{SMP}_r$ if and only if $r \geq p^*$.
\end{corollary}

\begin{proof}
(1) Assume first that $(\ell_p, L_p([0,1]))$  has the $\mathrm{SMP}_r$. Since $\ell_p$ embeds isometrically into $L_p([0,1])$ (see \cite[Proposition 6.4.18]{albiac}), we get  from  Proposition \ref{prop1}(1) that $(\ell_p, \ell_p)$ has the $\mathrm{SMP}_r$. By Theorem \ref{teoexlp}, we get that $r \geq p^*$.

Conversely, suppose that $r \geq p^*$. Let $T: \ell_p \to L_p([0,1])$ be a bounded linear operator admitting a weakly $r$-singular maximizing sequence $(x_n)_n \subset \ell_p$. If $(x_n)_n$ were weakly null, then by the Bessaga-Pelczynski selection principle \cite[Proposition 1.3.10]{albiac}, there exists a subsequence $(x_{n_k})_k$ equivalent to a block basic sequence with respect to the canonical basis $(e_k)_k$ of $\ell_p$. In particular, it follows that $(x_{n_k})_k$ is equivalent to $(e_k)_k$ (see, e.g., \cite[p. 426]{danieleduardo}). As $r \geq p^*$, the canonical basis of $\ell_p$ is weakly $r$-summable. Hence $(x_{n_k})_k$ is weakly $r$-summable, contradicting the 
 weak $r$-singularity of $(x_n)_n$. This shows that $(x_n)_n$ cannot be weakly null. Now, since $(\ell_p, L_p([0,1]))$ has the $\mathrm{WMP}$ by \cite[Example 3]{han}, we get that $T$ attains its norm, proving that $(\ell_p, L_p([0,1]))$ has the $\mathrm{SMP}_r$.
    \end{proof}

Since $\ell_2$ is a $1$-complemented subspace of $L_2([0,1])$ (see \cite[Theorem 6.2.13 and Theorem 6.4.2]{albiac} or the proof of \cite[Theorem 3.2]{dantasjung}), we get from Corollary \ref{corlp1} and from Proposition \ref{prop1} that $(L_2([0,1]), L_2([0,1]))$ cannot have the $\mathrm{SMP}_r$ whenever $1 \leq r < 2$, while the same pair has the $\mathrm{WMP}$, since $L_2([0,1])$ is a Hilbert space (see \cite[p. 5]{arongarciapelteix}). 
 %%Dantas, Jung, and Martínez-Cervantes asked whether $(L_p([0,1]), L_q([0,1]))$ has the WMP for $1 \leq p \leq 2 \leq q < \infty$ and $p \neq q$ (see~\cite[Question~4.2]{dantasjung}). Although this question concerns the classical $WMP$, the preceding results show that the stronger properties $SMP_r$ fail in this range whenever $1\leq r<2$: Concerning the singular maximizing property, we have that the pair $(L_2([0,1]), L_2([0,1]))$ fails to have the $SMP_r$ whenever $1 \leq r < 2$ \begin{corollary} \label{corlp}    Let $1 \leq p \leq 2 \leq q < \infty$. If $1\leq r < 2$, then the pair $(L_p([0,1]), L_q([0,1]))$ fails to have the $SMP_r$.\end{corollary} \begin{proof}     For the sake of contradiction, we assume that $(L_p([0,1]), L_q([0,1]))$ has the $SMP_r$. Since $\ell_q$   embeds isometrically into $L_q([0,1])$ (see \cite[Proposition 6.4.18]{albiac}) and \end{proof}

We recall that the pair $(\ell_p, c_0)$ has the $\mathrm{WMP}$ for every $1 < p < \infty$ (see \cite{garcia-lirola}). Our next result characterizes when this pair has the $\mathrm{SMP}_r$.

\begin{theorem} \label{teoc0}
    Let $1 < p < \infty$ and $1 \leq r < \infty$ be given.
    The pair $(\ell_p, c_0)$ has the $\mathrm{SMP}_r$ if and only if $r \geq p^*$.
\end{theorem}

\begin{proof}
    Assume first that $r < p^*$. Consider the operator  $T: \ell_{p}\to c_{0}$ defined by $$T(a_j)_j=((1-\frac{1}{j})a_{j})_j. $$ It is straightforward to check that $\norma{T} = 1$. We now show that $T$ does not attain its norm. Suppose, for the sake of contradiction, that there exists $a = (a_j)_j \in \ell_p$ such that $\norma{a}_p = 1$ and $\norma{Ta}_\infty = \displaystyle \sup_{j \in \N} |1 - \frac{1}{j}||a_j| = 1$. Since $\displaystyle |1 - \frac{1}{n}||a_n| \to 0$ when $n \to \infty$, exists $N \in \N$ such that $\displaystyle |1 - \frac{1}{n}||a_n| < \frac{1}{2}$ for every $n > N$, which implies that    $$ 1 = \displaystyle \sup_{j \in \N} |1 - \frac{1}{j}||a_j| = \sup_{j = 1,\dots, N} |1 - \frac{1}{j}||a_j|. $$  
      As $\{ |1 - \frac{1}{j}||a_j|: j = 1, \dots, N \}$ is a finite set, we can find $N_0 \in \{1, \dots, N\}$ such that $\displaystyle |1 - \frac{1}{N_0}||a_{N_0}| = 1$. Because $1 - \frac{1}{N_0} < 1$, we obtain $|a_{N_0}| > 1$, which is a contradiction with $\displaystyle \sum_{j=1}^\infty |a_j|^p = 1$. Thus, $T$ does not attain its norm. Notice, however, that the canonical basis  $(e_n)_n$ is a weakly $r$-singular maximizing sequence for $T$ for every $r < p^*$, and so $(\ell_{p},c_{0})$ fails to have $\mathrm{SMP}_r$.

      Conversely, we assume that $r \geq  p^*$.  Let $T: \ell_p \to c_0$ be a bounded linear operator admitting a weakly $r$-singular maximizing sequence $(x_n)_n \subset \ell_p$.
      As in the proof of Corollary \ref{corlp1}, $(x_n)_n$ cannot be weakly null. Now, since $(\ell_p, c_0)$ has the $\mathrm{WMP}$, we get that $T$ attains its norm, proving that $(\ell_p, c_0)$ has the $\mathrm{SMP}_r$.
\end{proof}

We now adapt some ideas from \cite{han} in order to obtain further sufficient conditions ensuring that a pair $(X,Y)$ has the $\mathrm{SMP}_p$. To do this, we recall that a pair of Banach spaces is said to have the {\bf property strict (M)} if for every contraction $T \in \mathcal{L}(X; Y)$, whenever $x\in X$ and $y \in Y$ satisfy $\norma{y} < \norma{x}$, it holds that $$ \limsup_{n} \norma{y + Tx_n} < \limsup_{n} \norma{x+x_n} $$
for every weakly null sequence $(x_n)_n \subset X$.

\begin{theorem} \label{more1}
Let $1 \leq p < \infty$, and let $X$ and $Y$ be Banach spaces such that  $B_X$ is a relatively weakly $p$-precompact set and the pair $(X, Y)$ has the property strict (M). Then, this pair has the $\mathrm{SMP}_p$.
\end{theorem} 

\begin{proof} To prove that $(X, Y)$ has the $\mathrm{SMP}_p$, let $T: X \to Y$ be a bounded linear operator with $\norma{T} = 1$ and let $(x_n)_n$ be a weakly $p$-singular maximizing sequence for $T$. By passing to a subsequence if necessary, we can assume that $(x_n - x)_n \in \ell_p^w(X)$ for some $x \in B_X$.
Since $(x_n)_n$ is weakly $p$-singular, $x \neq 0$. 
If $T$ fails to attain its norm, then $\norma{Tx} < \norma{x}$, and so the assumption implies that     \begin{align*}        1 & = \lim_n \norma{Tx_n} =  \limsup_{n} \norma{Tx + T(x_n-x)} < \limsup_{n} \norma{x+(x_n - x)} = \lim_n \norma{x_n} = 1,    \end{align*}    a contradiction. Thus, $T$ is a norm attaining operator, which proves that $(X, Y)$ has the $\mathrm{SMP}_p$.
\end{proof}

More examples of pairs with the $\mathrm{SMP}_p$ can be obtained with Theorem \ref{more1}.

\begin{examples}\rm {\rm (1)} Let $H$ and $K$ be Hilbert spaces. Then, the pair $(H, K)$ has the property strict (M). Indeed, if $T: H \to K$ is a contraction, $x \in H$ and $y \in K$ satisfy $\norma{y} <\norma{x}$, and $(x_{n})_{n}$ is a weakly null sequence in $H$, then 
\begin{align*}
    \limsup_{n}\norma{y + Tx_{n}}^{2}&=\limsup_{n}\inner{y+ Tx_{n}}{y + Tx_{n}} \\
        &=\limsup \inner{y}{y} + 2Re(\inner{y}{Tx_{n}})+ \inner{Tx_{n}}{Tx_{n}} \\
        &<\limsup_{n} \inner{x}{x}+ 2Re(\inner{y}{Tx_{n}})+ \inner{x_{n}}{x_{n}} \\
        &=\limsup_{n} \norma{x +x_{n}}^{2}.
\end{align*}
Moreover, since every Hilbert space is isometrically isomorphic to some $\ell_2(I)$ for some index set $I$ (see \cite[Exercise 5.8.13]{geraldo}), $B_H$ is relatively weakly $2$-precompact. Then Theorem \ref{more1} yields that $(H,K)$ has the $\mathrm{SMP}_r$ for every $r \geq 2$. Notice that for $H = K = L_2([0,1])$, $r = 2$ is sharp by our commentary after Corollary \ref{corlp1}.\\
{\rm (2)} It follows by Example 2.6(i) and Corollary 3.8(a) from \cite{garcia-lirola} that the pair $(X, c_0)$ has the $\mathrm{WMP}$. Then by \cite[Theorem 13]{han}, $(X, c_0)$ has the property strict (M). If we assume that $B_X$ is relatively weakly $p^*$-precompact, then Theorem \ref{more1} yields that $(X,c_{0})$ has the $\mathrm{SMP}_{p^*}$. In particular, we can consider $X = (\oplus_{n=1}^\infty \ell_1^n)_p$.
\end{examples}

The next result, which is a natural $p$-analogue of \cite[Theorem 3.5]{dantasjung}, presents new classes of pairs with the $\mathrm{SMP}_p$. We emphasize that the exponent $p$ is fixed in the statement below. In particular, although condition (3) concerns the $\mathrm{SMP}_r$ for every $1 \leq r < \infty$, the result characterizes the $\mathrm{SP}_p$ of $Y$ and does not assert that $Y$ has the $\mathrm{SP}_r$ for every $r$.

\begin{theorem}\label{p-schuriff}
For a Banach space $Y$ and $1 < p < \infty$, the following are equivalent: \\
{\rm (1)} $Y$ has the $\mathrm{SP}_p$. \\
{\rm (2)} For every Banach space $X$ such that $B_X$ is relatively weakly $p$-precompact, every bounded linear operator $T: X \to Y$ is norm-attaining. \\
{\rm (3)} For every $1 \leq r < \infty$ and every Banach space $X$ such that $B_X$ is relatively weakly $p$-precompact, the pair $(X, Y)$ has the $\mathrm{SMP}_r$. \\
{\rm (4)} For every Banach space $X$ such that $B_X$ is relatively weakly $p$-precompact, the pair $(X, Y)$ has the $\mathrm{WMP}$.
\end{theorem}

\begin{proof}(1)$\Rightarrow$(2)  Assume that $Y$ has the $\mathrm{SP}_p$. If $X$ is a Banach space such that $B_X$ is relatively weakly $p$-precompact, then every bounded linear operator $T: X \to Y$ is compact. Indeed, given a sequence $(x_n)_n$ in $B_X$, there exist a subsequence $(x_{n_k})_k$ and $x \in B_X$ such that $(x_{n_k} - x)_k \in \ell_p^w(X)$. As $T$ is bounded and $Y$ has the $\mathrm{SP}_p$, we get that $Tx_{n_k} \to Tx$ in $Y$, proving that $T$ is compact. Now, since $X$ is reflexive and every bounded linear operator from $X$ into $Y$ is compact, we conclude that every bounded linear operator from $X$ into $Y$ is norm-attaining.

(2)$\Rightarrow$(3) follows easily and (3)$\Rightarrow$(4) follows from Proposition \ref{prop1}.

(4)$\Rightarrow$(1) Assume that $Y$ fails to have the $\mathrm{SP}_p$. So, we can take a normalized weakly $p$-summable sequence $(y_n)_n \subset Y$. Notice that the bounded linear operator    $T: \ell_{p^{*}} \to Y$ given by $T((a_{n})_{n})=\displaystyle \sum\limits_{n=1}^{\infty} a_{n}y_{n}$ is a non-compact operator as it maps the canonical basis $(e_n)_n$ of $\ell_{p^*}$ into the normalized sequence $(y_n)_n$. Now, it follows from \cite[Theorem B]{sheldon} that there exists a non-norm attaining operator from $\ell_{p^*}$ into $Y$, and so the pair $(\ell_{p^{*}}\oplus_{\infty}\mathbb{R},Y)$ fails to have $\mathrm{WMP}$ (see \cite[Main Theorem(b)]{dantasjung}). %%By Proposition \ref{prop1}(4), we get that $(\ell_{p^{*}}\oplus_{\infty}\mathbb{R},Y)$ fails to have the $SMP_p$. 
To conclude the proof we check that $B_{\ell_{p^{*}}\oplus_{\infty}\mathbb{R}}$ is a relatively weakly $p-$precompact set. To see this, let $((a_{n})_{n},(\lambda_{n})_{n})\subset B_{\ell_{p^{*}}\oplus_{\infty}\R}$. On the one hand, as $(\lambda_{n})_n \subset[-1,1]$, there exists a convergent subsequence $\lambda_{n_k} \to \lambda$ in $\R$. Thus, we can extract an $p$-summable subsequence of $(\lambda_{n_{k}}-\lambda)_{k}$ that will also be denoted by $(\lambda_{n_{k}}-\lambda)_{k}$. On the other hand, since $B_{\ell_{p^*}}$ is a relatively weakly $p$-precompact and $(a_{n_k})_k \subset B_{\ell_{p^{*}}}$, there exist a weakly $p-$convergent subsequence of $(a_{n_{k}})_k$. For convenience, we denote this subsequence by  $(a_{n_{k}})$ and let $a$ be its weak $p$-limit, that is $(a_{n_k} - a)_k$ is a weakly $p$-summable sequence of $\ell_{p^*}$. We claim that $((a_{n_k}, \lambda_{n_{k}}) - (a, \lambda))_k$ is a weakly $p$-summable sequence.    Indeed, given $\varphi \in (\ell_{p^{*}}\oplus_{\infty} \mathbb{R})^{*} = \ell_{p}\oplus_{1}\mathbb{R}$ (see \cite[Exercise 4.5.13]{geraldo}), we can write $\varphi = (b,\gamma)$ for some $b \in \ell_{p}$ and $\gamma \in \mathbb{R}$. Then \begin{align*}        \left(\sum_{k=1}^{\infty}|\varphi(a_{n_{k}}-a,\lambda_{n_{k}}-\lambda)|^{p}\right)^{\frac{1}{p}}            &=\left(\sum_{k=1}^{\infty}|b(a_{n_{k}}-a) + \gamma(\lambda_{n_{k}}-\lambda)|^{p}\right)^{\frac{1}{p}} \\             &\leq\left( \sum_{k=1}^{\infty}|b(a_{n_{k}}-a)|^{p}\right)^{\frac{1}{p}} + \left(\sum_{k=1}^{\infty}|\gamma(\lambda_{n_{k}}-\lambda)|^{p}\right)^{\frac{1}{p}} \\            &< \infty.    \end{align*}    Therefore, \(((a_{n_k},\lambda_{n_k}) - (a,\lambda))_k\) is weakly \(p\)-summable. 
\end{proof}

We notice that one of the implications in Theorem \ref{p-schuriff} is false for $p =1$. To see this recall from \cite[p. 45]{castillo} that $B_X$ is relatively weakly $1$-precompact if and only if $X$ is finite dimensional. In this case, 
the pair $(X, Y)$ has the $\mathrm{SMP}_1$ for every Banach space $Y$.

\section{The weak $p$-singular maximizing property for adjoint operators}

In \cite{garcia-lirola}, García-Lirola and Petitjean introduced weak$^\ast$ maximizing properties for operators defined on dual spaces. The notion relevant for our purposes is the following: a pair of dual Banach spaces $(X^\ast,Y^\ast)$ has the {\bf weak$^*$-to-weak$^*$ maximizing property} if every adjoint operator $T^*: X^* \to Y^*$ is norm-attaining whenever there exists a non-weak$^*$ null maximizing sequence for $T^*$.

%%We begin this section recalling the two following maximizing-type properties introduced by García-Lirola and Petitjean in \cite{garcia-lirola}:\\ {\rm (i)} The pair $(X^*, Y)$ is said to have the {\bf weak$^*$ maximizing property} ($W^*MP$, in short) if a bounded linear operator $T: X^* \to Y$ is norm-attaining whenever there exists a non-weak$^*$ null maximizing sequence for $T$. \\ {\rm (ii)} The pair $(X^*, Y^*)$ has the {\bf weak$^*$ to weak$^*$ maximizing property} if an adjoint operator $T^*: X^* \to Y^*$ is norm-attaining whenever there exists a non-weak$^*$ null maximizing sequence for $T^*$.

Next, we record some basic relationships between this property, the $\mathrm{SMP}_p$ for adjoint operators introduced in Definition \ref{defsmppadj}, and the usual $\mathrm{SMP}_p$ on pairs of dual spaces.

%%%\medskip

%%%Since weakly $p$-summable and weak$^*$ $p$-summable sequences coincide in dual spaces (see, e.g., \cite[p. 2]{fourie}), the notion of weakly $p$-singular sequence already covers the weak$^*$ setting. Thus, no separate weak$^*$ version of the $SMP_p$ is needed for arbitrary operators on dual spaces. For adjoint operators, however, this restriction leads to a distinct property. Following \cite{garcia-lirola}, we introduce the corresponding adjoint version as follows.

%%\begin{definition} \label{adj1} Let $1 \leq p < \infty$.  A pair of dual Banach spaces $(X^*, Y^*)$ is said to have the \textbf{$SMP_p$ for adjoint operators} if every adjoint operator $T^*:X^*\to Y^*$ admitting a weakly $p$-singular maximizing sequence is norm-attaining. \end{definition} Next proposition lists a few standard facts about this new property: 

\begin{proposition}\label{adj2}
Let $X$ and $Y$ be Banach spaces and let $1\leq p<\infty$. \\
{\rm (1)} If the pair $(X^*,Y^*)$ has the $\mathrm{SMP}_p$, then it has the $\mathrm{SMP}_p$ for adjoint operators. \\
{\rm (2)} If the pair $(X^*,Y^*)$ has the $\mathrm{SMP}_p$ for adjoint operators, then it has the $\mathrm{SMP}_q$ for adjoint operators for every $q>p$.\\
{\rm (3)} If the pair $(X^*,Y^*)$ has the $\mathrm{SMP}_p$ for adjoint operators, then it has the weak$^*$-to-weak$^*$ maximizing property. 
\end{proposition}

\begin{proof}
Items (1) and (2) are immediate from the definitions and from the fact that, if $p<q$, then every weakly $p$-summable sequence is weakly $q$-summable. Hence every weakly $q$-singular sequence is weakly $p$-singular.

(3) Let $T^*:X^*\to Y^*$ be an adjoint operator admitting a non-weak$^*$ null maximizing sequence $(x_n^*)_n$. Passing to a subsequence if necessary, there exist $x\in X$ and $\varepsilon>0$ such that $|x_n^*(x)|\geq \varepsilon$ holds
for every $n\in\N$. Therefore $(x_n^*)_n$ has no weak$^*$ $p$-summable subsequence. Since weakly $p$-summability and weak$^*$ $p$-summability coincide in dual spaces (see \cite[p. 2]{fourie}), $(x_n^*)_n$ is weakly $p$-singular. By the $\mathrm{SMP}_p$ for adjoint operators, $T^*$ attains its norm.
\end{proof}

The converse of Proposition \ref{adj2}(1) holds if we assume that $X$ is reflexive. Indeed, since $X$ is reflexive, every bounded linear operator 
$S: X^* \to Y^*$ is weak$^*$-to-weak$^*$ continuous, and so an adjoint operator. The following example, however, shows that this implication does not hold in general: 

\begin{example} \label{adj3} \rm  Let $1<q<\infty$. The pair $(\ell_1^*,\ell_q^*)$ has the $\mathrm{SMP}_r$ for adjoint operators for every $1\leq r<\infty$, but it fails the $\mathrm{SMP}_r$ for every $1\leq r<\infty$. Indeed, every bounded linear operator from $\ell_q$ into $\ell_1$ is compact by Pitt's theorem, and so all such operators are norm-attaining as $\ell_q$ is reflexive. Consequently, every adjoint operator $T^*: \ell_1^* \to \ell_q^*$ attains its norm (see \cite[Lemma 5.1]{garcia-lirola}).  Thus $(\ell_1^*,\ell_q^*)$ has the $\mathrm{SMP}_r$ for adjoint operators for every $r \geq 1$. On the other hand, 
since $\ell_1^* = \ell_\infty$ is not reflexive, the pair $(\ell_1^*, \ell_q^*)$ cannot have the $\mathrm{SMP}_r$ for any $r \geq 1$.
\end{example}

We can easily notice that the converse implications in items (2) and (3) of Proposition \ref{adj2} are false. Indeed, since $L_2([0,1])$ is reflexive, every bounded linear operator $L_2([0,1]) \to L_2([0,1])$ is an adjoint operator. Thus, the pair $(L_2([0,1]), L_2([0,1]))$ has the 
weak$^*$-to-weak$^*$ maximizing property, but as it was shown in Section 2, this pair has the $\mathrm{SMP}_r$ if and only if $r \geq 2$.

Now, we recall that a pair $(X, Y)$ has the {\bf property strict (sM$^*$)} if for every contraction $T \in \mathcal{L}(X; Y)$, whenever $x^*\in X^*$ and $y^*\in Y^*$ satisfy $\norma{x^*} < \norma{y^*}$, it holds that 
$$\limsup_\alpha \norma{x^*+T^*y_\alpha^*} < \limsup_\alpha \norma{y^*+y_\alpha^*}$$
for every bounded weak$^*$ null net $(y_\alpha^*)_\alpha\subset Y^*$. 

\begin{proposition}\label{more-dual1}
Let $1\leq p<\infty$, and let $X$ and $Y$ be Banach spaces such that $B_{Y^*}$ is a relatively weakly $p$-precompact set and the pair $(X, Y)$ has the property strict (sM$^*$). Then $(Y^*,X^*)$ has the $\mathrm{SMP}_p$ for adjoint operators.
\end{proposition}

\begin{proof}
 Let $T: Y^{*} \to X^{*}$ be an adjoint operator. Without loss of generality, we may assume that $\norma{T}=1$. Suppose $T$ admits a weakly $p$-singular maximizing sequence $(y_{n}^{*})_{n}$, since $B_{Y^{*}}$ is relatively weakly $p$-precompact and passing to a subsequence if necessary there exists $0 \neq y^{*} \in B_{Y^{*}}$ such that $(y_{n}^{*}-y^{*})_{n} \in \ell_{p}^{w}(Y^{*})$. If $T$ fails to attain its norm, then $\norma{Ty^{*}}<\norma{y^{*}}$. By the assumption, $$1= \lim_{n\to \infty} \norma{Ty_{n}^{*}}=\limsup_{n}\norma{Ty^{*}+T(y_{n}^{*}-y^{*})}<\limsup_{n}\norma{ y^{*}+ (y_{n}^{*}-y^{*})}=1,$$ and therefore $T$ must attain its norm.
\end{proof}

To present examples, we first recall some terminology. The {\bf modulus of asymptotic uniform smoothness} of $X$ is, for each $t > 0$, 
$$ \overline{\rho}_X(t) = \sup_{\norma{x} = 1} \inf_{\dim X/Y < \infty} \sup_{y \in S_Y} \norma{x+ty} - 1, $$
and the {\bf modulus of asymptotic uniform convexity} of $X$ is, for each $t > 0$, 
$$ \overline{\delta}_X  (t) = \inf_{x \in S_X} \sup_{\dim X/Y < \infty} \inf_{y \in S_Y} \norma{x+ty} - 1. $$
If $X$ is a dual space, we can consider different moduli by taking all finite-codimensional
weak$^*$ closed subspaces $Y$ of $X$, and these moduli are denoted by¯$\overline{\rho}_{X}^*$ and $\overline{\delta}_{X^*}^*$, respectively.

\begin{example} \label{more-dual2} \rm 
Considering $Y = (\bigoplus_{i=1}^{\infty}\ell_{\infty}^{i})_{p^*}$, we have $Y^* = (\bigoplus_{i=1}^{\infty}\ell_{1}^{i})_{p}$. Since $Y$ is reflexive, we have from \cite[Example 2.6(i)]{garcia-lirola} that
$ \overline{\delta}_{Y^*}^*(t) = \overline{\delta}_{Y^*}(t) = (1+t^p)^{1/p} - 1 $
for every $t > 0$. 
Now, let $X$ be a Banach space such that $\overline{\rho}^*_{X^*}(t) \leq \overline{\delta}_{Y^*}(t)$ for every $t > 0$. 
In this case, we can apply Theorem 24 and Proposition 21 of \cite{han} to obtain that  $(X, Y)$ has the property strict (sM$^*$). Since $B_{Y^*}$ is relatively weakly $p$-precompact, we obtain by Proposition \ref{more-dual1} that $(Y^*, X^*)$ has the $\mathrm{SMP}_p$ for adjoint operators. 
%%\textcolor{blue}{ {\rm (1) } Let $2\leq p<\infty$ the Banach space $X_{p}^{*}=\left(\bigoplus_{i=1}^{\infty} \ell_{1}^{i}\right)_{p}$ is a dual space with its unit ball relatively weakly $p$-precompact satisfying property (sO$^{*}$), consequently $(X_{p}^{*},X_{p}^{*})$ satisfies (sO$^*$) and also having property (sM)$^{*}$ then it satisfy strict (sM)$^*$ and finally $SMP_{p}$.} \textcolor{red}{Combinar resultados de Garcia-Lirola e Petitjean - Ex 2.6 - com resultados do Han - Teo 24, Prop 21.} \\ \textcolor{blue}{{\rm (2)} Let $X, X_{p}^{*}$ be Banach spaces, being $X$ a dual reflexive space with Schauder basis satisfying an upper $p$-estimate with constant one, then the pair $(X,X_{p}^{*})$ satisfies property strict (sM$^{*}$) and consequently the $SMP_{p}$}
\end{example}

Our final objective in this section is to provide conditions on the Banach spaces $X$ and $Y$ under which every bounded linear operator
$T:X\to Y$ attains its norm if and only if the pair $(Y^\ast,X^\ast)$ has the $\mathrm{SMP}_q$ for adjoint operators. This will, in turn,
yield conditions ensuring that, for given $1\leq p,q<\infty$, the pair $(X,Y)$ has the $\mathrm{SMP}_p$ if and only if $(Y^\ast,X^\ast)$ has the $\mathrm{SMP}_q$ for adjoint operators. The latter equivalence is
closely related to \cite[Question 5.12]{garcia-lirola}. To this end, we will need the following Dunford-Pettis-type properties.

\begin{definition} \label{dppdef} \rm
    Let $1 \leq p, q \leq \infty$ be given. \\
    {\rm (i)} A Banach space $Y$ is said to have the {\bf sequential (p,q)-Dunford-Pettis property} ($\mathrm{sDPP}_{(p,q)}$, for short) if $y_n^*(y_n) \to 0$ for every $(y_n)_n \in \ell_p^w(Y)$ and every $(y_n^*)_n \in \ell_q^w( Y^*)$. For $p = \infty$ (or $q = \infty$), we replace $\ell_p$ (or $\ell_q$) by $c_0$ (see \cite[Definition 1.8]{gab2024}). \\
    {\rm (ii)} A Banach space $Y$ is said to have the {\bf Dunford-Pettis* property of order $p$} ($\mathrm{DP}^*\mathrm{P}_p$, for short) if  $y_n^*(y_n) \to 0$ for every $(y_n)_n \in \ell_p^w(Y)$ and every $(y_n^*)_n \in c_0^{w^*}(Y^*)$. For $p=\infty$, we replace $\ell_p$ by $c_0$ (see \cite[Theorem 2.4]{zeefou14}).
\end{definition}

Before proceeding we list some observations regarding the Dunford-Pettis properties defined in \ref{dppdef}.

\begin{remarks}    \rm 
(1) The $\mathrm{sDPP}_{(\infty, \infty)}$ coincides with the classical Dunford-Pettis property ($\mathrm{DPP}$, for short) \cite[p. 340]{alip}, and the $\mathrm{DP}^*\mathrm{P}_{\infty}$ coincides with the Dunford-Pettis* property ($\mathrm{DP}^*\mathrm{P}$, for short) \cite{cargalou}.\\
(2) If $Y$ is a Grothendieck space, i.e. every weak$^*$ null sequence in $Y^*$ is weakly null, then $Y$ has the $\mathrm{sDPP}_{(p, \infty)}$ if and only if $Y$ has the $\mathrm{DP}^*\mathrm{P}_{p}$. Without assuming that $Y$ is a Grothendieck space, it is well known that $\mathrm{sDPP}_{(\infty, \infty)} = \mathrm{DPP}$ is different than the $\mathrm{DP}^*\mathrm{P}_{\infty} = \mathrm{DP}^*\mathrm{P}$. \\
%%(3) If $Y$ is a dual Banach space with no copy of $\ell_1$, then $Y^*$ has the $\mathrm{sDPP}_{(p, \infty)}$ if and only if $Y$ has the $\mathrm{sDPP}_{(\infty, p)}$ (see \cite[Corollary 3.2]{chenchavez}). \\
(3) We recall from \cite{coarse} that a Banach space $Y$ is said to have the coarse $p-\mathrm{DP}^*$ property if for every relatively weakly compact subset $A \subset Y$ and every bounded linear operator $T: Y \to \ell_p$, it holds that $T(A)$ is relatively compact in $\ell_p$. We claim that every Banach space $Y$ with the coarse $p-\mathrm{DP}^*$ property has the $\mathrm{sDPP}_{(\infty, p)}$. Indeed, if $(y_n)_n$ is a weakly null sequence in $Y$ and $(y_n^*)_n$ is a weakly $p$-summable sequence in $Y^*$, then $A = \conj{y_k}{k \in \N}$ is a relatively weakly compact subset of $Y$ and $T(y):=(y_n^*(y))_n$ defines a bounded linear operator from $Y$ into $\ell_p$. Thus, our assumption yields that $T(A) = \conj{(y_n^*(y_k))_{n}}{k \in \N}$ is a relatively compact subset of $\ell_p$. By \cite[Exercise 15, p. 168]{alip}, we get that $$ s_n:= \sup \conj{\sum_{i=n}^\infty |y_i^*(y_k)|^p}{k \in \N} \to 0$$ when $n \to \infty$. Hence $ \displaystyle  |y_n^*(y_n)|^p \leq \sum_{i=n}^\infty |y_i^*(y_n)|^p \leq s_n \to 0, $ which proves that $Y$ has the $\mathrm{sDPP}_{(\infty, p)}$. \\
(4) A stronger property was defined by Karn and Sinha. Following \cite{karnsinha2}, let $1 \leq p < \infty$, $1 \leq q \leq \infty$, a Banach space $X$ is said to have the $(p,q)$-$\mathrm{DPP}$ if given $(x_n)_n \in \ell_q^w(X)$ and $(x_n^*)_n \in \ell_p^w(X^*)$, it follows that $((x_k^*(x_n))_k)_n\in \ell_p^s(\ell_q)$, i.e.
$$  \sum_{n=1}^\infty  \left ( \sum_{k=1}^\infty |x_k^*(x_n)|^{q} \right)^{p/q} =  \sum_{n=1}^\infty \norma{(x_k^*(x_n))_k}_q^p < \infty. $$
This implies that
$ \displaystyle \lim_{n \to \infty} \sum_{k=1}^\infty |x_k^*(x_n)|^{q} \longrightarrow 0. $
Since, for each $n \in \N$, $|x_n^*(x_n)|^q \leq \sum_{k=1}^\infty |x_k^*(x_n)|^{q}$, we get that
$ \displaystyle \lim_{n \to \infty} |x_n^*(x_n)| = 0, $
proving that $X$ has the $\mathrm{sDPP}_{(q,p)}$. Notice that the indices $p$ and $q$ are inverted in the definitions of $(p,q)$-$\mathrm{DPP}$ and $\mathrm{sDPP}_{(p,q)}$.
\end{remarks}

\begin{lemma}\label{adj4}
Let $X$ be a Banach space such that $B_X$ is relatively weakly $p$-precompact with $1\leq p\leq \infty$, let $Y$ be a Banach space, and let $T:X\to Y$ be a nonzero bounded linear operator.\\
{\rm (1)} If $Y$ has the $\mathrm{sDPP}_{(p,q)}$ with $1\leq q<\infty$, then every maximizing sequence for $T^*:Y^*\to X^*$ is weakly $q$-singular. \\
{\rm (2)} If $Y$ has the $\mathrm{DP}^*\mathrm{P}_p$, then every maximizing sequence for $T^*:Y^*\to X^*$ is non-weak$^*$ null.
\end{lemma}

\begin{proof}
Let $(y_n^*)_n \subset Y^*$ be a maximizing sequence for $T^*$. Since $X$ is reflexive, we can find $(x_n)_n \subset S_X$ such that $T^*y_n^*(x_n) = \norma{T^*y_n^*}$ for every $n \in \N$. Besides, since $B_X$ is relatively weakly $p$-precompact, there exists a subsequence $(x_{n_{k}})_k$ and a vector $x \in B_X$ such that $(x_{n_k} - x)_k \in \ell_p^w(X)$ (or  $c_0^w(X)$, if $p = \infty$). The continuity of $T$ implies that $(Tx_{n_k} - Tx)_k \in \ell_p^w(Y)$ (or  $c_0^w(Y)$, if $p = \infty$).

To prove (1), we assume, for the sake of contradiction, that $(y_n^*)_n$ is not weakly $q$-singular. So, passing to a subsequence $(y_{n_{k}})_{k}$ that is weakly $q$-summable. 
As $Y$ has the $\mathrm{sDPP}_{(p,q)}$, it follows that $y_{n_k}^*(Tx_{n_k} - Tx) \to 0$. Thus,
\begin{align*}
    \norma{T^*} & = \lim_{k \to \infty} [\norma{T^*y_{n_k}^*}   - T^*y_{n_k}^*(x)] \\
    & = \lim_{k \to \infty} [T^*y_{k}^*(x_{n_k}) - T^*y_{n_k}^*(x)] = \lim_{k\to \infty} y_{n_k}^*(Tx_{n_k} - Tx) = 0,
\end{align*}
which is a contradiction since $T \neq 0$. Item (2) follows from the same argument.
\end{proof}

\begin{theorem}\label{adj5}
Let $1\leq p\leq\infty$ and $1\leq q<\infty$. Suppose that $B_X$ is relatively weakly $p$-precompact and that $Y$ has the $\mathrm{sDPP}_{(p,q)}$. Then the following assertions are equivalent:\\{\rm (a)} Every adjoint operator $T^*:Y^*\to X^*$ attains its norm.\\
{\rm (b)} Every bounded linear operator $T:X\to Y$ attains its norm.\\
{\rm (c)} The pair $(Y^*,X^*)$ has the $\mathrm{SMP}_q$ for adjoint operators.
\end{theorem}

\begin{proof}
The implication {\rm (a)}$\Rightarrow${\rm (b)} follows from \cite[Proposition 2.1]{mishra} or \cite[Lemma 5.1]{garcia-lirola}. The implication {\rm (b)}$\Rightarrow${\rm (c)} follows from the fact that if $T:X\to Y$ attains its norm, then $T^*:Y^*\to X^*$ also attains its norm; see, e.g., \cite[p. 3]{mishra}.

Finally, assume {\rm (c)} and let $T:X\to Y$ be a nonzero bounded linear operator. By Lemma \ref{adj4}, every maximizing sequence for $T^*$ is weakly $q$-singular. Hence $T^*$ admits a weakly $q$-singular maximizing sequence, and by {\rm (c)} it attains its norm. Therefore every adjoint operator $T^*:Y^*\to X^*$ attains its norm, proving {\rm (a)}.
\end{proof}

Under the assumptions that $B_X$ is relatively weakly $p$-precompact, $Y$ has the $\mathrm{sDPP}_{(p,q)}$ and $B_{Y^*}$ is relatively weakly $q$-precompact, it follows straightforward that every bounded linear operator from $X$ into $Y$ attains its norm. So, we obtain the following:

\begin{corollary}\label{adj6}
Let $1\leq p,q<\infty$. Suppose that $B_X$ is relatively weakly $p$-precompact, that $B_{Y^*}$ is relatively weakly $q$-precompact, and that $Y$ has the $\mathrm{sDPP}_{(p,q)}$. Then, the pair $(X,Y)$ has the $\mathrm{SMP}_p$ if and only if the pair $(Y^*,X^*)$ has the $\mathrm{SMP}_q$ for adjoint operators.
\end{corollary}

%%\begin{proof}We first assume that $(X,Y)$ has the $\mathrm{SMP}_p$. Let $T:X\to Y$ be a nonzero bounded linear operator such that $T^*:Y^*\to X^*$ admits a weakly $q$-singular maximizing sequence $(y_n^*)_n$. Since $X$ is reflexive, we can find a sequence $(x_n)_n \subset S_X$ such that $y_n^*(Tx_n) = \norma{T^*y_n^*}$ for every $n \in \N$. Thus, it follows that $(x_n)_n$ is a maximizing sequence for $T$. Now, since $B_X$ is relatively weakly $p$-precompact, by passing to a subsequence if necessary, we may assume that $(x_n - x)_n$ is a weakly $p$-summable sequence for some $x \in B_X$. On the other hand, as $B_{Y^*}$ is relatively weakly $q$-precompact, passing to a further subsequence, we may also assume that $(y_n^* - y^*)_n$ is weakly $q$-summable for some $y^* \in B_{Y^*}$. As we are assuming that $(y_n^*)_n$ is weakly $q$-singular, $y^* \neq 0$. Now, since $Y$ has the $\mathrm{sDPP}_{(p,q)}$, $(y_n^* - y^*)(Tx_n - Tx) \rightarrow 0$, and so  $y_n^*(Tx_n)\longrightarrow y^*(Tx).$ Hence $$ \norma{T} = \lim_{n \to \infty} \norma{T^*y_n^*} = \lim_{n \to \infty} y_n^*(Tx_n) = y^*(Tx), $$ which implies that $x \neq 0$. Therefore, $(x_n)_n$ is weakly $p$-singular, and our assumption yields that $T$ attains its norm. Consequently, $T^*$ attains its norm, which proves that $(Y^*,X^*)$ has the $\mathrm{SMP}_q$ for adjoint operators. Assuming that $(Y^*, X^*)$ has the $\mathrm{SMP}_q$, we get from Theorem \ref{adj5} that every bounded linear operator from $X$ into $Y$ attains its norm, so $(X, Y)$ has the $\mathrm{SMP}_p$. \end{proof}

The following examples show that the relative weak $p$-precompactness of $B_X$ and the relative weak $q$-precompactness of $B_{Y^*}$ cannot, in general, be omitted from Corollary \ref{adj6}.

\begin{example}\label{contraexadj1} \rm 
(1) Since $c_0$ is not reflexive, the pair $(c_0, \K)$ fails the $\mathrm{SMP}_p$ for every $1\leq p<\infty$. On the other hand, the pair $(\K^*, c_0^*)$ has the $\mathrm{SMP}_q$ for adjoint operators for every $1\leq q<\infty$,
since every operator with finite-dimensional domain attains its norm. Notice that this example does not contradict Corollary \ref{adj6} since $B_X$ is not relatively weakly $p$-precompact for every $1 < p < \infty$. 

(2) Take $X = \ell_p$ with $1 < p < \infty$ and $Y = c_0$. So, $B_{X}$ is relatively weakly $p^*$-precompact, $Y$ has the $\mathrm{sDPP}_{(p^*,q)}$ for every $1 \leq q < \infty$ since $Y^* = \ell_1$ has the Schur property, and $B_{Y^*}$ is not relatively weakly $q$-precompact for every $q$.  It follows from Theorem \ref{teoc0} that  $(X, Y)$ has the $\mathrm{SMP}_{p^*}$. Nevertheless, the pair $(Y^*, X^*)$ fails to have the $\mathrm{SMP}_q$ for adjoint operators. Indeed, considering the non-norm attaining operator  $T: \ell_{p}\to c_{0}$ defined by $\displaystyle T(a_j)_j=((1-\frac{1}{j})a_{j})_j $ from the proof of Theorem \ref{teoc0}, its adjoint $T^*: c_0^* \to \ell_p^*$ does not attain its norm. Since 
$ T^*(b_j)_j = ((1- \frac{1}{j})b_j)_j $
for every $(b_j)_j \in \ell_1 = c_0^*$, we get that $(e_n)_n \subset S_{\ell_1}$ is a maximizing sequence for $T^*$. Since $(e_n)_n$ is weakly $q$-singular in $\ell_1$ for every $1 \leq  q < \infty$, we obtain that $(Y^*, X^*)$ fails to have the $\mathrm{SMP}_q$ for adjoint operators.
\end{example}

The same proof of Theorem \ref{adj5} proves the following:

\begin{theorem} \label{adj7}
    Let $1 \leq p \leq \infty$. If $B_X$ is relatively weakly $p$-precompact and $Y$ has the $\mathrm{DP}^*\mathrm{P}_p$, then the following are equivalent. \\
    {\rm (a)} Every adjoint operator $T^{*}:Y^{*} \to X^{*}$ attains its norm.\\
       {\rm (b)} Every operator $T:X \to Y$ attains its norm.\\
        {\rm (c)} $(Y^{*},X^{*})$ has the weak$^*$-to-weak$^*$ maximizing property.
\end{theorem}

The case $p = \infty$ of Theorem \ref{adj7} coincides with \cite[Proposition 5.10]{garcia-lirola}. For $1 \leq p < \infty$, the class of admissible spaces $X$ becomes more restrictive, whereas the class of spaces $Y$ for which the result holds is enlarged.

Under the assumptions that $B_X$ is relatively weakly $p$-precompact, $Y$ has the $\mathrm{DP}^*\mathrm{P}_p$ , and $B_{Y^*}$ is weak$^{*}$ sequentially compact, we obtain that all bounded linear operators from $X$ into $Y$ attain  their norms. Thus, the following holds:

\begin{corollary}\label{adj8}
Let $1\leq p<\infty$. Suppose that $B_X$ is relatively weakly $p$-precompact and that $Y$ is a Banach space with the $\mathrm{DP}^*\mathrm{P}_p$ such that $B_{Y^*}$ is weak$^{*}$ sequentially compact. Then $(X,Y)$ has the $\mathrm{SMP}_p$ if and only if $(Y^*,X^*)$ has the weak$^*$-to-weak$^*$ maximizing property.
\end{corollary}

The following examples show that the assumptions on $B_X$ and on the $\mathrm{DP}^*\mathrm{P}_p$ in Corollary \ref{adj8} cannot, in general, be omitted.

\begin{examples} \label{adj9}
    \rm (1) Since $c_0$ is not reflexive, the pair $(c_0, \K)$ cannot have the $\mathrm{SMP}_p$ for every $1 \leq p < \infty$. However, $(\K^*, c_0^*)$ has the weak$^*$-to-weak$^*$ maximizing property since $\K^*$ is finite dimensional. Notice that this example does not contradict Corollary \ref{adj8} because $B_{c_0}$ is not relatively weakly $p$-precompact. \\
    (2) It follows from \cite[Remark 3.2]{garcia-lirola} that the pair $(\R \oplus_\infty \ell_2, c_0)$ does not have the $\mathrm{WMP}$, so it cannot have the $\mathrm{SMP}_2$. However, $(c_0^*, (\R \oplus_\infty \ell_2)^*)$ has the weak*-to-weak* maximizing property by \cite[Corollary 5.3]{garcia-lirola}. This example does not contradict Corollary \ref{adj8} since $c_0$ does not have the $DP^*P_2$. Notice, however, that $B_{c_0^*}$ is weak* sequentially compact and that $B_{\R \oplus_\infty \ell_2}$ is relatively weakly $2$-precompact by the same argument presented in the proof of Theorem \ref{p-schuriff}.
\end{examples}

We do not know whether the $\mathrm{sDPP}_{(p,q)}$ assumption in Corollary \ref{adj6} or the weak$^*$ sequential compactness of $B_{Y^*}$ in Corollary \ref{adj8} can be omitted. %%%Notice that a counterexample in either case would provide a negative answer to \cite[Question 5.12]{garcia-lirola}.

\section{The $p$-convergent perturbation property}

We begin this section by proving some basic properties of the $p$-$\mathrm{CPP}$.

\begin{proposition}\label{pcpp1}
Let $1\leq p<\infty$. \\
{\rm (1)} If $(X,Y)$ has the $\mathrm{SMP}_p$, then $(X,Y)$ has the $p$-$\mathrm{CPP}$.\\
{\rm (2)} If $(X, Y)$ has the $p$-$\mathrm{CPP}$, then it has the $q$-$\mathrm{CPP}$ for every $q > p$. \\
{\rm (3)} Suppose that $(X, Y)$ satisfies the $p$-$\mathrm{CPP}$. If $E$ is a $1$-complemented subspace of $X$ and $F$ is a closed subspace of $Y$, then $(E,F)$ has the $p$-$\mathrm{CPP}$. 
\end{proposition}

\begin{proof} (1) Assume for the sake of contradiction that $(X, Y)$ does not have the $p$-$\mathrm{CPP}$. Thus, there exist a bounded linear operator $T: X \to Y$ and a $p$-convergent operator $K : X \to Y$ such that $T+K$ does not attain its norm and $\norma{T} < \norma{T+K}$. If $(x_n)_n$ is a maximizing sequence for $T+K$, as we are assuming that $(X, Y)$ has the $\mathrm{SMP}_p$, $(x_n)_n$ cannot be weakly $p$-singular. So, passing to a subsequence if necessary, we may assume that $(x_n)_n$ is weakly $p$-summable. Since $K$ is $p$-convergent, we have $Kx_n \to 0$, and so
$$ \norma{T+K} = \lim_{n \to \infty} \norma{(T+K)x_n} \leq \norma{T}, $$
a contradiction with $\norma{T} < \norma{T+K}$.

(2) Assume now that $(X, Y)$ satisfies the $p$-$\mathrm{CPP}$ for some $1 \leq p < \infty$ and let  $p < q \leq \infty$. Consider a bounded linear operator $T: X \to Y$ and a $q$-convergent operator (resp. compact operator when $q=\infty$) $K : X \to Y$ such that $\norma{T+K} > \norma{T}$. Since $q > p$, it follows that $K$ is a $p$-convergent operator. Thus, the assumption yields that $T+K$ is a norm attaining operator, proving that $(X, Y)$ has the $q$-$\mathrm{CPP}$. 

(3) Assume that $(X,Y)$ satisfies the $p$-CPP for some $1\leq  p<\infty $. Let $T: E \to F$ be a bounded linear operator and let $K:E \to F$ be a $p$-convergent operator such that $\norma{T}< \norma{T+K}$. Since $E$ is a $1$-complemented subspace of $X$ there exists a projection $P: X \to E$ with $\norma{P}=1$ and let $\iota: F \to Y$ be the inclusion map. Define $\widetilde{T} := \iota \circ T \circ P: X \to Y$ and $\widetilde{K}:= \iota \circ K \circ P : X \to Y$. Let us prove that 
$\norma{T}= \norma{\widetilde{T}}$. Indeed, the definition of the norm yields that
$$\norma{\widetilde{T}}= \sup_{x \in B_{X}}\norma{\widetilde{T}x}\geq \sup_{x \in B_{E}} \norma{\iota \circ T x}= \norma{T}, $$
and the reverse inequality follows from 
$\norma{\widetilde{T}}\leq \norma{\iota}\norma{T} \norma{P}= \norma{T}$ since $\norma{P} = 1 = \norma{\iota}$. The same argument proves that  $\norma{K}= \norma{\widetilde{K}}$ and that $\norma{T+K} = \norma{\widetilde{T} + \widetilde{K}}$. Hence
$\norma{\widetilde{T}}< \norma{\widetilde{T}+ \widetilde{K}}$. Moreover, it is easy to see that $\widetilde{K}$ is a $p$-convergent operator, so our assumption implies that $\widetilde{T}+ \widetilde{K}$ attains its norm at some $x_{0} \in S_{X}$. Therefore
$$\norma{(T+K)(Px_{0})}=\norma{\iota \circ (T+K)Px_{0}}= \norma{(\widetilde{T}+ \widetilde{K})x_{0}}=\norma{\widetilde{T}+ \widetilde{K}}= \norma{T+K},$$
and we are done.
\end{proof}

Now, we present a few examples of pairs enjoying (or failing) the $p$-$\mathrm{CPP}$.

\begin{examples} \label{pcpp2}\rm  Let $1 < p < \infty$ and $1 \leq q, r < \infty$ be given. \\ (1) {\it  The pair $(\ell_p, \ell_q)$ has the $r$-$\mathrm{CPP}$ if and only if $q < p$ or $r \geq p^*$}. If $q < p$ or $r \geq p^*$, then $(\ell_p, \ell_q)$ has the $\mathrm{SMP}_r$ by Theorem \ref{teoexlp}, and so it has the $r$-$\mathrm{CPP}$.
Conversely, assume that $q \geq p$ and that $r < p^*$. Thus, $\ell_p$ has the $\mathrm{SP}_r$, which implies that the operator $T: \ell_p \to \ell_q$ defined by    $$T(x)=\left(\frac{n x_{n}}{2(n+1)}\right)_{n}$$    is $r$-convergent. Nevertheless, it follows that $\norma{T + T} > \norma{T}$ and that $T+ T$ does not attain its norm by the proof of Theorem \ref{teoexlp}. \\
(2) The pair $(\ell_p, c_0)$ has the $r$-$\mathrm{CPP}$ if and only if $r \geq p^*$.  If $r \geq p^*$, it follows from Theorem \ref{teoc0} that $(\ell_p, c_0)$ has the $\mathrm{SMP}_r$, so it has the $r$-$\mathrm{CPP}$. Conversely, assuming that $r < p^*$, we get that $\ell_p$ has the $\mathrm{SP}_r$, so the operator $T: \ell_p \to c_0$ given by 
$$ T(x)=\left(\frac12\left(1-\frac1n\right)a_n\right)_n $$ is $r$-convergent. However, it follows that $\norma{T + T} > \norma{T}$ and that $T+T$ does not attain its norm by the proof of Theorem \ref{teoc0}. \\
{\rm (3)} It follows from \cite[Proposition 3.6]{dantasjung} that there exists a Banach space $E_p$ isomorphic to $\ell_p$ such that the pair $(E_p, c_0)$ does not have the $\mathrm{WMP}$. So, it cannot have the $\mathrm{SMP}_{p^*}$. We claim that $(E_p, c_0)$ has the $p^*$-$\mathrm{CPP}$.
To see this, let $T: E_p \to c_0$ be a bounded linear operator and let $K: E_p \to c_0$ be a $p^*$-convergent operator such that $\norma{T+K} > \norma{T}$. Since $E_p$ is isomorphic to $\ell_p$, $B_{E_p}$ is relatively weakly $p^*$-precompact, $K$ must be compact. However, since $(E_p, c_0)$ has the $\mathrm{CPP}$ by \cite[Proposition 3.6]{dantasjung}, we get that $T+K$ attains its norm. 
\end{examples}

Jung, Martínez-Cervantes and Rueda-Zoca proved that for every Banach space $X$ there exist a Banach space $Y$, a bounded linear operator $T: X \to Y$, and a rank-one operator $R: X\to Y$ such that $\norma{T} < \norma{T+R}$ and $T+R$ fails to attain its norm (see \cite[Theorem 2.1]{jung}). This implies, in particular, that $(X, Y)$ fails to have the $\mathrm{CPP}$. By adapting their argument, we obtain the following.

\begin{theorem} \label{teocpp} Let $ 1 < p < \infty$. If $X$ is infinite-dimensional and $B_X$ is a relatively weakly $p$-precompact set, then there exist a Banach space $Y$ such that $B_{Y}$ is relatively weakly $p$-precompact, a bounded linear operator $T: X \to Y$, and a rank one operator $R: X \to Y$ such that $\norma{T} < \norma{T+R}$ and $T+R$ fails to attain its norm.   
\end{theorem}

\begin{proof} Since $X$ is  reflexive, there are a bounded linear operator $T: X \to \ell_\infty$  and a rank-one operator $K: X \to \ell_{\infty}$ such that $\norma{T+K} > \norma{T} = 1$ and $T+K$ is not a norm-attaining operator (see \cite[Theorem 2.1]{jung}). Proceeding as in the proof of \cite[Theorem 1.1]{jung}, there exists a reflexive Banach space $Z = \widetilde{T}(X) \oplus \widetilde{K}(X)$, where  $\widetilde{T}, \, \widetilde{K} \in \mathcal{L}(X; \ell_{\infty} \oplus_{\infty}X)$ are given by 
$$\widetilde{T}(x)=(T(x),x) \quad \text{ and } \quad \widetilde{K}(x)=(K(x),0). $$
Thus, we have that $\widetilde{T}: X \to Z$ is a bounded a linear operator and $\widetilde{K}: X \to Z$ is a rank-one operator such that $\norma{\widetilde{T}} < \norma{\widetilde{T}+ \widetilde{K}}$ and $\widetilde{T}+ \widetilde{K}$ fails to attain its norm. 

To conclude our proof it is enough to show that $B_Z$ is relatively weakly $p$-precompact. Indeed, 
given $(y_{n})_{n} \subset B_{Z}$, there are $(u_n)_n, \, (v_n)_n \subset X$ such that 
$$y_{n}= \widetilde{T}(u_{n}) + \widetilde{K}(v_{n}) = (Tu_n + Kv_n, u_n)$$ for every $n \in \N$. On the one hand, for each $n \in \N$
$$ \norma{u_n}  \leq \max \{ \norma{Tu_n + Kv_n}, \, \norma{u_n} \} =  \norma{y_n} \leq 1, $$
so $(u_n)_n$ is contained in $B_X$ which is assumed to be relatively weakly $p$-precompact. Thus, there exist a subsequence $(u_{n_k})_k$ of $(u_n)_n$ and $u \in B_X$ such that $(u_{n_k}-u)_k$ is weakly $p$-summable in $X$. Therefore $(\widetilde{T}(u_{n_{k}}-u))_{k} $ is a weakly $p$- summable subsequence. On the other hand, since $(\widetilde{K}(v_{n_k}))_k$ is a bounded sequence in the one-dimensional space $\widetilde{K}(X)$, we can extract a weakly $p$-convergent subsequence of $(\widetilde{K}(v_{n_k}))_k$. Without loss of generality, we say that $(\widetilde{K}(v_{n_k}) - \widetilde{K}(v))_k$ is a weakly $p$-summable sequence.
 Now, setting $y= \widetilde{T}(u) + \widetilde{K}(v)$, we get that $(y_{n_k} - y)_k$ is a weakly $p$-summable sequence.
\end{proof}

Theorem \ref{teocpp} shows that the $p$-$\mathrm{CPP}$ may fail even when both $B_X$ and $B_Y$ are relatively weakly $p$-precompact, and that this failure can already be witnessed by a rank-one perturbation.

We now turn to the adjoint version of the $p$-$\mathrm{CPP}$ introduced in Definition \ref{adpcpp}. To clarify its connection with the classical setting, recall that Han and Kim \cite{han} defined the adjoint compact perturbation property (ACPP) by requiring $\norma{T} = \norma{T}_e := \displaystyle \inf_{K \in \mathcal{K}(X; Y)} \norma{T-K}$ for every $T \in \mathcal{L}(X; Y)$ whose adjoint $T^*$ does not attain its norm. When $X$ is reflexive, the ACPP is equivalent to the CPP. We next give parallel perturbative characterizations of the ACPP and the $p$-ACPP.

\begin{theorem} \label{acppchar1}
Let $X$ and $Y$ be Banach spaces and $1 \leq p < \infty$. The pair $(X, Y)$ has the  $p$-$\mathrm{ACPP}$ (resp. $\mathrm{ACPP}$) if and only if $T^* + K^*$ attains its norm whenever $T, K \in \mathcal{L}(X; Y)$ satisfy $\norma{T} < \norma{T+K}$ and $K^*$ is $p$-convergent (resp. $K$ is compact). 
\end{theorem}

\begin{proof}
    Assume first that $(X, Y)$ has the $p$-ACPP, and let $T, K: X \to Y$ be bounded linear operators such that $K^*$ is $p$-convergent and $\norma{T} < \norma{T+K}$. Thus, 
    \begin{align*}
        \norma{T+K}_{e,p}^{\rm adj} & = \inf \conj{\norma{T+K-S}}{S^* \text{ is $p$-convergent}} \leq \norma{T} < \norma{T+K},
    \end{align*}
    and so the assumption implies that $T^* + K^* = (T+K)^*$ attain its norm.

    To prove the converse, we assume, for the sake of contradiction that there exists $T \in \mathcal{L}(X; Y)$ such that $T^*$ does not attain its norm and $\norma{T}_{e,p}^{\rm adj} < \norma{T}$. Thus, there exists $K: X \to Y$ such that $K^*$ is $p$-convergent and $\norma{T-K} < \norma{T}$. This implies that
    $  \norma{T-K} < \norma{(T-K) + K}   $
    and $(T-K)^* + K^* = T^*$ does not attain its norm, a contradiction.

    The proof of the compact case is analogous. 
\end{proof}

We conclude this paper with some remarks regarding the $p$-ACPP.

\begin{remarks} \rm 
(1) Let $X$ be a Banach space such that $B_X$ is relatively weakly $p$-precompact. If $(X, Y)$ has the $p$-ACPP, then it has the $p$-CPP. Indeed, if $T: X \to Y$ is a bounded linear operator and $K: X \to Y$ is a $p$-convergent operator satisfying $\norma{T} <\norma{T+K}$, then as $B_X$ is relatively weakly $p$-precompact, $K$ is compact. 
Consequently, $K^{*}$ is compact, and thus $p$-convergent. Since $(X,Y)$ has the $p$-ACPP, $T^* + K^*$ attains its norm. Finally, since $X$ is reflexive, $T+K$ also attains its norm.  \\
(2) The assumption that $B_X$ is relatively weakly $p$-precompact in item (1) cannot be omitted. Indeed, if $X$ is a non-reflexive Banach space, there exists a non-norm attaining linear functional $0 \neq x^*: X \to \K$. 
Taking $T = 0$ and $K= x^*$, we get that $K$ is $p$-convergent,
$\norma{T} = 0 < \norma{T+K}$, and $T+K = K$ does not attain its norm. This proves that $(X, \K)$ fails the $p$-CPP. However, since every adjoint operator from $\K^*$ into $X^*$ attains its norm, the pair $(X, \K)$ has the $p$-ACPP. \\
(3) Let $1\leq p<\infty$. If $(Y^*,X^*)$ has the $\mathrm{SMP}_p$
for adjoint operators, then $(X,Y)$ has the $p$-ACPP.  \\ 
{\rm (4)} Let $1 <p < \infty$ and $1 \leq q,r< \infty$. The pair $(\ell_p,\ell_q)$ has the $r$-ACPP if and only if $p > q$ or $r \geq q$. Indeed, suppose first that $q > 1$. Since $\ell_p$ and $\ell_q$ are reflexive, every operator from $\ell_q^*$ into $\ell_p^*$ is an adjoint operator. Hence, by Theorem \ref{acppchar1}, the pair $(\ell_p, \ell_q)$ has the $r$-ACPP if and only if $(\ell_q^*, \ell_p^*)$ has the $r$-CPP. By Examples 
\ref{pcpp2}(1), this occurs if and only if $p^* < q^*$ or $r \geq (q^*)^*$ = q, which is equivalent to $p > q$ or $r \geq q$.

If $q=1$, every operator from $\ell_p$ into $\ell_1$ is compact, hence norm-attaining. Consequently, every adjoint operator from $\ell_1^*$ into $\ell_p^*$ attains its norm, and $(\ell_p, \ell_1)$ has the $r$-ACPP for every $1 \leq r < \infty$, consistently with the condition $p > q$. \\
(5) The pair $(\ell_p, c_0)$ fails the $r$-ACPP for every $1 \leq r < \infty$. Indeed, considering the operator $D: \ell_p \to c_0$ given by $D((a_j)_j) = ((1-1/j)a_j)_j$, we have that $\norma{D} = 1$, $D^*$ does not attain its norm and that $D^*$ is $r$-convergent for every $1 \leq r < \infty$. Taking $T = 0$ and $K = D$ in Theorem \ref{acppchar1}, we are done. \\
(6) It follows from item (5) above that the pair $(\ell_2, c_0)$ does not have the $2$-ACPP, however this pair has the $2$-CPP by Examples \ref{pcpp2}(2). This examples shows that the $p$-CPP does not imply the $p$-ACPP. 
\end{remarks}

\noindent O. I. Blanco A.\\
Centro de Matem\'atica, Computa\c c\~ao e Cogni\c c\~ao \\
Universidade Federal do ABC \\
09.210-580 -- Santo André - SP -- Brazil.  \\
e-mail: oscarblanco3322@gmail.com

\medskip

\noindent V. C. C. Miranda\\
Departamento de Matemática\\
Instituto de Ciências Matemáticas e de Computação \\
Universidade de São Paulo \\
13566-590 -- São Carlos - SP -- Brazil  \\
e-mail: viniciusmiranda@icmc.usp.br

\end{document}